\documentclass{amsart}
\usepackage{amsfonts}
\usepackage{amsmath}
\usepackage{amssymb}
\usepackage{amsthm}

\newtheorem{theorem}{Theorem}
\newtheorem{prop}{Proposition}
\newtheorem{lemma}{Lemma}

\newtheorem{exmp}{Example}
\newtheorem{cor}{Corollary}

\begin{document}

\title{On the distance between linear codes}
\author{Mariusz Kwiatkowski, Mark Pankov}
\subjclass[2000]{}
\address{Department of Mathematics and Computer Science, 
University of Warmia and Mazury, S{\l}oneczna 54, Olsztyn, Poland}
\email{mkw@matman.uwm.edu.pl, pankov@matman.uwm.edu.pl}

\maketitle
\begin{abstract}
Let $V$ be an $n$-dimensional vector space over 
the finite field consisting of $q$ elements
and let $\Gamma_{k}(V)$ be the Grassmann graph 
formed by $k$-dimensional subspaces of $V$, $1<k<n-1$.
Denote by $\Gamma(n,k)_{q}$ the restriction of $\Gamma_{k}(V)$
to the set of all non-degenerate linear $[n,k]_{q}$ codes.
We show that for any two codes the distance in
$\Gamma(n,k)_{q}$ coincides with the distance in $\Gamma_{k}(V)$ 
only in the case when
$n<(q+1)^2+k-2$,
i.e. if $n$ is sufficiently large then for some pairs of codes
the distances in the graphs $\Gamma_{k}(V)$ and $\Gamma(n,k)_{q}$ are distinct.
We describe one class of such pairs.
\end{abstract}

\section{Introduction}
The distance between two vertices in a connected simple graph $\Gamma$
is the smallest number $m$ such that there is a path of length $m$,
i.e. a path consisting of $m$ edges, which joins these vertices
\cite[Section 15.1]{DD}.
If $\Gamma'$ is a connected subgraph of $\Gamma$
then the distance between two vertices of $\Gamma'$
may be greater than the distance between these vertices in $\Gamma$.

Consider the Grassmann graph $\Gamma_{k}(V)$
formed by $k$-dimensional subspaces of an $n$-dimensional vector space $V$.
Two $k$-dimensional subspaces are adjacent vertices of $\Gamma_{k}(V)$
if their intersection is $(k-1)$-dimensional.
We suppose that $1<k<n-1$,
since for $k=1,n-1$ any two distinct vertices in the Grassmann graph are adjacent.
The Grassmann graph is connected and the distance between two 
$k$-dimensional subspaces $X,Y$ is equal to $k-\dim(X\cap Y)$.
See \cite{BCN-book,P1,P2} for more information.

The set of all $k$-dimensional subspaces complementary 
to a given $(n-k)$-dimen\-sional subspace can be identified with 
the set of all  matrices of dimension $(n-k)\times k$. 
The restriction of $\Gamma_{k}(V)$ to this set is isomorphic to the graph,
where two matrices $A$ and $B$ are adjacent 
if the rank of $A-B$ is equal to $1$.
This matrix graph is connected and the distance between two matrices $A,B$ is equal 
to the rank of $A-B$, i.e. it coincides with
the distance between the corresponding subspaces in the Grassmann graph \cite{Wan}.

In this paper we suppose that $V$ is a vector space over 
the finite field consisting of $q$ elements and 
denote by $\Gamma(n,k)_{q}$ the restriction of the Grassmann graph to 
the set of all non-degenerate linear $[n,k]_{q}$ codes \cite{TVN}. 
It is not difficult to prove that this graph is connected and 
the distance between two vertices is equal or  grater by one than 
the distance between these vertices in the Grassmann graph.
Our first result (Theorem 1) says that 
all distances in the graph $\Gamma(n,k)_{q}$ coincide with 
the distances in the Grassmann graph
if and only if 
$$n<(q+1)^2+k-2.$$
So, if $n$ is sufficiently large then
for some pairs of linear $[n,k]_{q}$ codes the distance in $\Gamma(n,k)_{q}$ 
is greater  than the distance in the Grassmann graph.
A class of such pairs will be described by the second result (Theorem 2).

\section{The graph of linear codes}
Let $\mathbb{F}_{q}$ be the finite field consisting of $q$ elements.
Consider the $n$-dimensional vector space 
$$V=\underbrace{{\mathbb F}_{q}\times\dots \times{\mathbb F}_{q}}_{n}$$
over this field.
The standard base of $V$ is formed by the vectors 
$$e_{1}=(1,0,\dots,0),\dots,e_{n}=(0,\dots,0,1).$$
We denote by $c_{i}$ the $i$-th coordinate functional
$$(x_{1},\dots,x_{n})\to x_{i}.$$
and write $C_{i}$ for its kernel. 

A {\it linear} $[n,k]_{q}$ {\it code} is a $k$-dimensional subspaces of $V$.
Following \cite{TVN} we restrict ourself to the case of non-degenerate linear codes.
A linear $[n,k]_{q}$ code $C\subset V$ is {\it non-degenerate} if
the restriction of every $c_{i}$ to $C$ is non-zero, i.e. $C$ is not contained in any $C_{i}$.

Let ${\mathcal G}_{k}(V)$ be the Grassmannian consisting of 
all $k$-dimensional subspaces of $V$. 
The {\it Grassmann graph} $\Gamma_{k}(V)$
is the graph whose vertex set is ${\mathcal G}_{k}(V)$ and 
two $k$-dimensional subspaces are adjacent vertices of this graph
if their intersection is $(k-1)$-dimensional.
The graph is connected and the distance between $X,Y\in {\mathcal G}_{k}(V)$
in this graph is equal to 
$$k-\dim(X\cap Y)=\dim(X+Y)-k.$$
We denote this distance by $d(X,Y)$.
The number 
$$m(n,k)=k-\min\{k,n-k\}$$
is the smallest dimension for the intersection of two $k$-dimensional subspaces
and the diameter of $\Gamma_{k}(V)$ is equal to $\min\{k,n-k\}$.

The number of $k$-dimensional subspaces of $V$ is equal to the Gaussian coefficient
$$\genfrac{\lbrack}{\rbrack}{0pt}{}{n}{k}_q=
\frac{(q^n-1)(q^{n-1}-1)\cdots(q^{n-k+1}-1)}{(q-1)(q^2-1)\cdots(q^k-1)};$$
in particular, the number of $1$-dimensional subspaces 
and the number of hyperplanes in $V$ are equal to
$$\genfrac{[}{]}{0pt}{}{n}{1}_q=\genfrac{\lbrack}{\rbrack}{0pt}{}{n}{n-1}_q=
[n]_{q}=\frac{q^n-1}{q-1}.$$
All (non-degenerate) linear $[n,k]_{q}$ codes form the set 
$${\mathcal C}(n,k)_{q}=
{\mathcal G}_{k}(V)\setminus \left(\bigcup_{i=1}^{n}{\mathcal G}_{k}(C_{i}) \right).$$
Using the inclusion-exclusion principle, we can show that this set contains precisely
$$\sum_{i=0}^{n-k}(-1)^i\binom{n}{i}\genfrac{\lbrack}{\rbrack}{0pt}{}{n-i}{k}_q$$
elements.
We write $\Gamma(n,k)_{q}$ for the restriction of the Grassmann graph $\Gamma_{k}(V)$ 
to ${\mathcal C}(n,k)_{q}$.

\begin{prop}
The graph  $\Gamma(n,k)_{q}$ is connected.
\end{prop}

\begin{proof}
Let $X,Y\in  {\mathcal C}(n,k)_{q}$. 
We suppose that $\dim X\cap Y$ is not greater than $k-2$,
i.e. $X$ and $Y$ are non-adjacent.
Consider the case when one of our codes, say $X$, contains a vector $x$ of weight $n$
(the weight $||x||$ is the number of non-zero coordinates in $x$).
We take any hyperplane $A\subset Y$ containing 
$X\cap Y$.
If $x\not\in X\cap Y$ then we denote by $Y_{1}$
the $k$-dimensional subspace spanned by $A$ and $x$.
In the case when $x\in X\cap Y$, we define $Y_{1}$ as 
the $k$-dimensional space spanned by $A$ and a certain vector in $X\setminus Y$.
For each of these cases we have 
$$\dim(X\cap Y_{1})=\dim (X\cap Y)+1$$
and $Y,Y_{1}$ are adjacent.
Step by step, we construct a path
$$Y,Y_{1},\dots,Y_{i}=X,\;\;\;i=d(X,Y)$$
in $\Gamma_{k}(V)$ such that every $Y_{j}$ contains the vector $x$.
The latter guarantees that all $Y_{j}$ belong to ${\mathcal C}(n,k)_{q}$.
So, $X$ and $Y$ are connected in $\Gamma(n,k)_{q}$.

If $X\cup Y$ does not contain vectors of weight $n$ then 
we consider any $k$-dimensional subspace $Y'$ spanned by 
a vector of weight $n$ and
a hyperplane of $Y$ containing $X\cap Y$.
Then $Y'$ belongs to ${\mathcal C}(n,k)_{q}$ and $X\cap Y=X\cap Y'$.
By the arguments given above, $\Gamma(n,k)_{q}$ contains a path
$$Y',Y_{1},\dots,Y_{i}=X,\;\;\;i=d(X,Y)=d(X,Y').$$
Note that $Y$ and $Y'$ are adjacent.
\end{proof}

For any $X,Y\in {\mathcal C}(n,k)_{q}$ we denote by $d_{c}(X,Y)$
the distance between $X$ and $Y$ in the graph $\Gamma(n,k)_{q}$.
Arguments from the proof of Proposition 1 give the following.
\begin{cor}
Let $X,Y\in {\mathcal C}(n,k)_{q}$. Then
$$d_{c}(X,Y)=d(X,Y)$$
if $X\cup Y$ contains a vector of weight $n$;
otherwise, we have
$$d(X,Y)\le d_{c}(X,Y)\le d(X,Y)+1.$$
\end{cor}
 
The main results of the paper are the following statements.

\begin{theorem}
If 
\begin{equation}\label{eq1}
 n<(q+1)^{2}+k-2
\end{equation}
then
$$d_{c}(X,Y)=d(X,Y)$$ 
for all $X,Y\in {\mathcal C}(n,k)_{q}$.
\end{theorem}

\begin{theorem}
If $m$ is a number satisfying $m(n,k)\le m\le k-2$ and
$$n\ge [k-m]_q\cdot(q+1)+m$$
then there exist $X,Y\in {\mathcal C}(n,k)_{q}$ such that
$$d(X,Y)=k-m\;\mbox{ and }\; d_{c}(X,Y)=k-m+1.$$
\end{theorem}

If  \eqref{eq1} does not hold then for $m=k-2$
Theorem 2 implies the existence of  $X,Y\in {\mathcal C}(n,k)_{q}$
such that $d(X,Y)=2$ and $d_{c}(X,Y)=3$.
So, we get the following.

\begin{cor}
For every pair $X,Y\in {\mathcal C}(n,k)_{q}$
the distance in $\Gamma(n,k)_{q}$ is equal to the distance in $\Gamma_{k}(V)$ 
if and only if \eqref{eq1} holds.
\end{cor}

\section{Proof of Theorem 1}
Our proof is based on the following.

\begin{lemma}\label{lem1}
Let $X$ and $Y$ be non-adjacent elements of ${\mathcal C}(n,k)_{q}$. 
Suppose that 
there exist no $Z\in {\mathcal C}(n,k)_{q}$ satisfying the following conditions:
\begin{enumerate}
\item[(1)] $X$ and $Z$ are adjacent,
\item[(2)] $d(Z,Y)=d(X,Y)-1$.
\end{enumerate}
Then
$$n-m\ge [k-m]_q\cdot(q+1),$$
where $m=\dim(X\cap Y)$.
\end{lemma}

\begin{proof}
If $Z\in {\mathcal G}_{k}(V)$ satisfies $(1)$ and $(2)$ 
then it intersects $X$ in a hyperplane containing $X\cap Y$.
There are precisely $ [k-m]_q$ 
such hyperplanes in $X$.
Let $H$ be one of them.
By our assumption, 
every $Z \in {\mathcal G}_{k}(V)$ containing $H$ and satisfying $(2)$ 
is contained in at least one of the hyperplanes $C_i$. 
Since $Y\in{\mathcal C}(n,k)_{q}$, 
each $C_i$ intersects $Y$ in a hyperplane. 
It is clear that the minimal number of hyperplanes in $Y$ 
whose union coincides with $Y$ is equal to $q+1$.
This guarantees the existence of at least $q+1$ distinct $C_i$ containing $H$.
Indeed, if this number is less than $q+1$
then there is a vector $x\in Y\setminus X$ such that 
the subspace spanned by $H$ and $x$
belongs to ${\mathcal C}(n,k)_{q}$ which contradicts our assumption. 

So, every hyperplane of $X$ containing $X\cap Y$ is contained in at least 
$q+1$ distinct $C_{i}$.
Since $X\in{\mathcal C}(n,k)_{q}$, there is no $C_i$ which contains two distinct hyperplanes of $X$.
Thus there are at least 
$$[k-m]_q\cdot(q+1)$$
distinct $C_{i}$ containing $X\cap Y$.

On the other hand, $\dim(X\cap Y)=m$ and there are at most $n-m$ distinct $C_{i}$
containing $X\cap Y$. This implies the required inequality.
\end{proof}

\begin{lemma}\label{lem2}
Let $X$ and $Y$ be non-adjacent elements of ${\mathcal C}(n,k)_{q}$
and let $m$ be the dimension of $X\cap Y$.  If 
$$n<  [k-m]_q\cdot(q+1)+m$$
then 
there exists $Z\in {\mathcal C}(n,k)_{q}$ adjacent to $X$ and such that
$$d(Z,Y)=d(X,Y)-1.$$
\end{lemma}

\begin{proof}
Follows directly from Lemma \ref{lem1}.
\end{proof}

\begin{lemma}\label{lem3}
For every number $m\le k-2$ we have
$$[k-m]_{q}(q+1)+m\ge (q+1)^{2}+k-2.$$
\end{lemma}

\begin{proof}
Consider the function 
$$f(x)=\frac{q+1}{q-1}(q^{k-x}-1)+x.$$
The inequality follows from the fact that
$$f'(x)=1-\ln(q)\frac{q+1}{q-1}q^{k-x}<0$$
if $x\le k-2$.
\end{proof}

Now we prove Theorem 1.
Let $X$ and $Y$ be non-adjacent elements of ${\mathcal C}(n,k)_{q}$.
Suppose that $\dim(X\cap Y)=m$.
Then 
$$d(X,Y)=k-m\ge 2,$$
i.e. $m\le k-2$.
By Lemma \ref{lem3}, the inequality
$$n< (q+1)^{2}+k-2$$
implies the inequality from Lemma \ref{lem2}.
Thus there exists $X_1\in {\mathcal C}(n,k)_{q}$ adjacent to 
$X$ and such that 
$$d(X_{1},Y)=d(X,Y)-1.$$
Applying the same arguments to $X_{1}$ and $Y$,
we obtain $X_2\in {\mathcal C}(n,k)_{q}$ adjacent to $X_1$ 
and satisfying
$$d(X_{2},Y)=d(X_{1},Y)-1=d(X,Y)-2.$$
Step by step, we construct a path 
$$X,X_{1},\dots,X_{k-m}=Y,$$
where every $X_i$ belongs to ${\mathcal C}(n,k)_{q}$. 
So, $d_{c}(X,Y)=k-m$.

\section{Proof of Theorem 2}
We consider two preliminary examples and prove Theorem 2 in several steps.
Our first example concerns the case when $q=2$, $k=2$ and $n=9$.

\begin{exmp}{\rm
Let $X$ and $Y$ be the linear $[9,2]_{2}$ codes  whose non-zero vectors are
$$\left[
\begin{array}{c}
v_1\\
v_2\\
v_{1}+v_{2}
\end{array}
\right]=\left[
\begin{array}{ccccccccc}
0&0&0&1&1&1&1&1&1\\
1&1&1&0&0&0&1&1&1\\
1&1&1&1&1&1&0&0&0\\
\end{array}
\right]$$
and 
$$\left[
\begin{array}{c}
u_1\\
u_2\\
u_{1}+u_{2}
\end{array}
\right]=\left[
\begin{array}{ccccccccc}
0&1&1&0&1&1&0&1&1\\
1&0&1&1&0&1&1&0&1\\
1&1&0&1&1&0&1&1&0\\
\end{array}
\right],$$
respectively.
Observe that 
$$v_{1}\in C_{1}\cap C_{2}\cap C_{3},\;\;v_{2}\in C_{4}\cap C_{5}\cap C_{6},\;\;
v_{1}+v_{2}\in C_{7}\cap C_{8} \cap C_{9}$$
and 
$$u_{1}\in C_{1}\cap C_{4}\cap C_{7},\;\;u_{2}\in C_{2}\cap C_{5}\cap C_{8},\;\;
u_{1}+u_{2}\in C_{3}\cap C_{6} \cap C_{9}.$$
This means that every pair of non-zero vectors $v\in X$ and $u\in Y$
is contained in a certain $C_{i}$, i.e.
there is no linear $[9,2]_{2}$ code which has non-zero intersections with both $X$ and $Y$.
So, $d(X,Y)=2$ and $d_{c}(X,Y)=3$.
}\end{exmp}

The above construction can be adapted to the case of an arbitrary finite field.

\begin{exmp}{\rm
Suppose that $k=2$ and $n=(q+1)^{2}$.
The number $q$ is an arbitrary prime power.
Let $X$ be the linear $[n,2]_{q}$ code generated by the following vectors
$$v_1=
(\,\underbrace{0,\dots,0}_{q+1},1,\dots,1\,)$$
and
$$v_2=
(\,\underbrace{1,\dots,1}_{q+1},\underbrace{0,\dots, 0}_{q+1},x_{0},\dots,x_{q-2}),$$
where 
$$x_{i}=\underbrace{-\alpha^{-i}, \dots, -\alpha^{-i}}_{q+1}\,,\;\;\;i\in \{0,\dots,q-2\}$$
and $\alpha$ is a primitive element of ${\mathbb F}_{q}$.
Other non-zero vectors of $X$ are scalar multiples of
$$v_{1},\,v_{2},\,v_{3}=v_{1}+v_{2},\,v_{4}= v_{1}+\alpha v_{2},
\dots, v_{q+1}=v_1+\alpha^{q-2} v_2.$$ 
For every $v_i$ the following coordinates are zero: 
$$1+(i-1)(q+1),\;2+(i-1)(q+1),\dots,q+1+(i-1)(q+1)$$
(this is obvious for $i=1,2$ and for other $i$ this coordinates are
$1-\alpha^{i-3}\cdot \frac{1}{\alpha^{i-3}}=0$).
Consider the linear $[n,2]_{q}$ code $Y$ generated by the following vectors
$$u_1=(\,\underbrace{y,\dots,y}_{q+1}\,), \mbox{ where } y=\underbrace{0,1,\dots, 1}_{q+1},$$
and
$$u_2=
(\,\underbrace{z,\dots,z}_{q+1}\,), \mbox{ where } 
z=(1,0,-1,-\alpha^{-1},\dots,-\alpha^{-(q-2)}).$$
Other non-zero vectors of $Y$ are scalar multiples of
$$u_{1},\,u_{2},\,u_{3}=u_{1}+u_{2},\,u_{4}= u_{1}+\alpha u_{2},
\dots, u_{q+1}=u_1+\alpha^{q-2} u_2.$$ 
We check that for every $u_j$
the following coordinates are zero: 
$$j,\;j+(q+1),\dots,j+q(q+1).$$
It is easy to see that for every pair $i,j$
the $2$-dimensional subspace spanned by $v_{i}$ and $u_{j}$
is contained in the hyperplane $C_l$ for 
$$l=(i-1)(q+1)+j.$$
We have $X\cap Y=0$, since in every vector of $X$ the first $q+1$ coordinates are coincident
and for every non-zero vector of $Y$ the same fails.
As in the previous example, $d(X,Y)=2$ and $d_{c}(X,Y)=3$.
}\end{exmp}

Now we consider the case when
$$2<k\le n-k\;\mbox{ and }\;n=[k]_{q}\cdot (q+1).$$
The $k$-dimensional vector space ${\mathbb F}^{k}_{q}$
contains precisely $[k]_{q}$ vectors whose first non-zero coordinate is $1$. 
We denote these vectors by $w_{1},\dots,w_{[k]_{q}}$
and write $G_{X}$ for the $(k\times n)$-matrix 
$$[X_{1}\dots X_{[k]_{q}}],$$
where every $X_{i}$, $i\in \{1,\dots,[k]_{q}\}$ is the matrix of dimension $k\times(q+1)$
whose columns are the vector $w_{i}$.
Let $y$ and $z$ be as in Example 2.
We define the $(k\times n)$-matrix 
$$G_{Y}=\left[
\begin{array}{cccccccccc}
y&0&0&\cdots&0&0&\cdots&0\\
z&y&0&\cdots&0&0&\cdots&0\\
0&z&y&\cdots&0&0&\cdots&0\\
\vdots&\vdots&\vdots&\ddots&\vdots&\vdots&\ddots&\vdots \\
0&0&0&\cdots&z&y&\cdots&y\\
0&0&0&\cdots&0&z&\cdots&z\\
\end{array}
\right] .$$

\begin{lemma}
If $X$ and $Y$ are the linear $[n,k]_{q}$ codes
whose generator matrices are $G_{X}$ and 
$G_{Y}$\footnote{If $C$ is the linear $[n,k]_{q}$ code spanned by vectors $v_{1},\dots,v_{k}$
then the corresponding generator matrix of $C$ is the $(k\times n)$-matrix 
whose rows are the vectors $v_{1},\dots,v_{k}$.}, respectively,
then 
$$d(X,Y)=k\;\mbox{ and }\;d_{c}(X,Y)=k+1.$$
\end{lemma}

To prove Lemma 4 we will use the following observation concerning hyperplanes in 
a $k$-dimensional vector space over ${\mathbb F}_{q}$.
Let $e_{1},\dots,e_{k}$ be a base of this  vector space.
There are precisely $q^{k-1}$ distinct hyperplanes which do not contain $e_{1}$;
every hyperplane of such type is spanned by vectors 
$e'_1,\dots,e'_{k-1}$, where $e'_i=e_{i+1}$ or $e'_i=e_1+a_{i}e_{i+1}$ for 
a certain non-zero scalar $a_{i}$.
For every $l\in \{2,\dots, k-1\}$ there are precisely $q^{k-l}$ hyperplanes which contain 
$e_1,\dots,e_{l-1}$ and do not contain $e_l$;
any of these hyperplanes is spanned by 
$e_1,\dots,e_{l-1}$ and vectors $e'_l,\dots,e'_{k-1}$
where $e'_i=e_{i+1}$ or $e'_i=e_l+a_{i}e_{i+1}$  and $a_i$ is a non-zero scalar.
At the end, we get the unique hyperplane containing $e_{1},\dots,e_{k-1}$. 
So, in this way we have listed all 
$$q^{k-1}+q^{k-2}+\dots +q+1=[k]_q$$
hyperplanes.

\begin{proof}[Proof Lemma 4]
As in Example 2, we have $X\cap Y=0$.
Indeed, in all vectors of $X$ the first $q+1$ coordinates are coincident
and for the non-zero vectors of $Y$ the same fails.
So, $d(X,Y)=k$.

Let $v_{1},\dots,v_{k}$ be the rows of $G_{X}$.
These vectors form a base of $X$.
Consider the hyperplane $H\subset X$ spanned by the vectors
$$v_1+a_2v_2,\;v_1+a_3v_3,\dots,v_1+a_kv_k,$$
where all $a_{i}$ are non-zero. 
There is $t=i(q+1)+1$ such that the columns of $G_X$
with the numbers $t,t+1,\dots,t+q$ are
$$\left[\begin{array}{c}
1\\
-a_2^{-1}\\
-a_3^{-1}\\
\vdots\\
-a_n^{-1}\\
\end{array}
\right]$$
The corresponding coordinates of  the vectors $v_1+a_iv_i$
are equal to $1-a_i\frac{1}{a_i}=0$.
This implies that $H$ is contained in 
the $q+1$ coordinate hyperplanes $C_t,C_{t+1},\dots,C_{t+q}$.

Using the above description of hyperplanes in $X$,
we show that for every hyperplane of $X$ there is $t=i(q+1)+1$
such that this hyperplane is contained in $C_t,C_{t+1},\dots,C_{t+q}$.

Let $u_{1},\dots,u_{k}$ be the rows of $G_{Y}$.
These vectors form a base of $Y$. We immediately see that vectors $u_{3},\dots,u_{k}$
are contained in the intersection of coordinate hyperplanes $C_1,C_2,\dots,C_{q+1}$
and, by Example 2, every vector from the subspace $\langle u_1,u_2\rangle$ 
is contained in one of these coordinate hyperplanes. 
Hence every vector of $Y$ is contained
in one of the coordinate hyperplanes $C_1,C_2,\dots,C_{q+1}.$
Similarly, we establish that all vectors of $Y$ are contained in 
the union of coordinate hyperplanes $C_t,C_{t+1},\dots,C_{t+q}$ for every $t=i(q+1)+1$.

Therefore, for every hyperplane of $X$ and every vector in $Y$
there is  a certain $C_{i}$ containing them, i.e.
there is no $Z\in {\mathcal C}(n,k)_{q}$
adjacent to $X$ and having a non-zero intersection with $Y$.
This means that $d_{c}(X,Y)=k+1$.
\end{proof}

\begin{lemma}
If $$2\le k\le n-k\;\mbox{ and }\;n=[k]_{q}\cdot (q+1)$$
then there exist $X,Y\in {\mathcal C}(n,k)_{q}$
such that 
$$d(X,Y)=k\;\mbox{ and }\;d_{c}(X,Y)=k+1.$$
\end{lemma}

\begin{proof}
The statement follows from Example 2 and Lemma 4.
\end{proof}

\begin{lemma}
If $$n=[k-m]_q\cdot (q+1)+m,$$
where $m$ satisfies $m(n,k)\le m\le k-2$, then there exist $X,Y\in {\mathcal C}(n,k)_{q}$
such that 
$$d(X,Y)=k-m\;\mbox{ and }\;d_{c}(X,Y)=k-m+1.$$
\end{lemma}

\begin{proof}
By Lemma 5, there exist linear $[n-m,k-m]_{q}$ codes $X',Y'$
such that 
$$X'\cap Y'=0\;\mbox{ and }\;d_{c}(X',Y')=k-m+1.$$
Let $G_{X'}$ and $G_{Y'}$ be the generator matrices of $X'$ and $Y'$, respectively.
Denote by $X$ and $Y$ the linear $[n,k]_{q}$ codes whose generator matrices are
$$
G_X=\left[
\begin{array}{cc}
G_{X'}&0\\
0&I_m\\
\end{array}
\right]\;\mbox{ and }\;
G_Y=\left[
\begin{array}{cc}
G_{Y'}&0\\
0&I_m\\
\end{array}
\right],$$
respectively.
It is clear that $\dim (X\cap Y)=m$, i.e. $d(X,Y)=k-m$.
Show that
$$d_{c}(X,Y)=k-m+1.$$
If $Z\in {\mathcal G}_{k}(V)$ is adjacent to $X$
and $d(Z,Y)=k-m-1$ then $Z\cap X$ is a hyperplane of $X$,
the subspace $Z\cap Y$ is $(m+1)$-dimensional and 
each of the subspaces contains $X\cap Y$.
Note that every subspace of $X$ or $Y$ containing $X\cap Y$
can be identified with a certain subspace of $X'$ or $Y'$, respectively.
The described above properties of the pair $X',Y'$ guarantee that
for every hyperplane of $X$ containing $X\cap Y$
and every vector in $Y\setminus X$ there is  a certain $C_{i}$ containing them.
This implies that $Z$ does not belong to ${\mathcal C}(n,k)_{q}$ 
and we get the claim.
\end{proof}

To complete the proof, we consider the general case when 
$$n\ge [k-m]_{q}\cdot (q+1)+m$$
and $m$ satisfies $m(n,k)\le m\le k-2$.
Let 
$$n'=[k-m]_q\cdot (q+1)+m.$$ 
Lemma 6 implies the existence of $X',Y'\in C(n',k)$ such that 
$$d(X',Y')=k-m\;\mbox{ and }d_{c}(X',Y')=k-m+1.$$
We denote their generator matrices by $G_{X'}$ and $G_{Y'}$
(respectively) and consider the linear $[n,k]_{q}$ codes 
$X$ and $Y$  corresponding to the extended generator matrices 
$$G_X=[G_{X'}\,\,\mathbf{1}]\;\mbox{ and }\;G_Y=[G_{Y'}\,\,\mathbf{1}],$$
where $\mathbf{1}$ is the matrix of dimension $k\times(n-n')$ whose elements are $1$.
The equality $d(X,Y)=k-m$ is obvious and we show that
$d_{c}(X,Y)=k-m+1$.

Every hyperplane $H\subset X$ containing $X\cap Y$
is an extension of a certain hyperplane $H'\subset X'$
containing $X'\cap Y'$. 
Similarly, every vector $y\in Y\setminus X$
is an extension of a vector $y'\in Y'\setminus X'$.  
Since $d_{c}(X',Y')=k-m+1$, there is a coordinate hyperplane
containing $H'$ and $y'$.
Then $H$ and $y$ are contained in the same coordinate hyperplane.
This implies the required inequality.

\end{document}